\documentclass[reqno]{amsart}
\usepackage{graphicx, amsmath,amsfonts,amsthm,amssymb, paralist,fullpage}

\newtheorem{theorem}{Theorem}[section]

\newtheorem{lemma}[theorem]{Lemma}

\theoremstyle{definition}

\theoremstyle{remark}
\newtheorem{remark}[theorem]{Remark}
\numberwithin{equation}{section}
\newcommand{\g}{\geqslant}

\newcommand{\RR}{\mathbb{R}}

\newcommand{\CC}{\mathbb{C}}

\newcommand{\NN}{\mathbb{N}}
\newcommand{\p}{\partial}

\newcommand{\les}{\leqslant}
\newcommand{\lesa}{\lesssim}

\newcommand{\mc}[1]{\mathcal{#1}}

\newcommand{\lr}[1]{ \langle #1 \rangle}

\DeclareSymbolFont{bbold}{U}{bbold}{m}{n}
\DeclareSymbolFontAlphabet{\mathbbold}{bbold}

\begin{document}

\title{Long Range Scattering for the Cubic Dirac equation on $\RR^{1+1}$}%

\author[T.~Candy]{Timothy Candy}%
\address[T.~Candy]{Universit\"at Bielefeld, Fakult\"at f\"ur Mathematik,
  Postfach 100131, 33501 Bielefeld, Germany}%
\email{tcandy@math.uni-bielefeld.de}%

\author[H.~Lindblad]{Hans Lindblad}
\address[H.~Lindblad]{Department of Mathematics, Johns Hopkins University, 404 Krieger Hall, 3400 N. Charles Street, Baltimore, Maryland 21218} 
\email{lindblad@math.jhu.edu}

\thanks{T.C. acknowledges support from the German Research
  Foundation via Collaborative Research Center 701. H.L. is partially supported by NSF grant DMS--1237212. }



%\date{}%
%\dedicatory{}%
%\commby{}%
% ----------------------------------------------------------------
\begin{abstract}
We show that the cubic Dirac equation, also known as the Thirring model, scatters at infinity to a linear solution modulo a phase correction.
\end{abstract}
\maketitle
% ----------------------------------------------------------------
\section{Introduction}

We consider the cubic Dirac equation (also known as the Thirring model)
    \begin{equation}\begin{split}\label{eqn - thirring model}
      (\p_t + \p_x) u &= i v + i |v|^2 u \\
      (\p_t - \p_x) v &= i u + i |u|^2 v\\
    \end{split}\end{equation}
with data $u(1)=f$, $v(1)=g$ where $u, v : \RR^{1+1} \rightarrow \CC$. This model was introduced by Thirring in \cite{Thirring1958} and describes the self interaction of a Dirac field.

It is known that solutions exist globally in time, provided that the data $f, g \in L^2$ \cite{Candy2010}. With regards to regularity, this is sharp in the sense that the $L^2_x$ norm is scale invariant (at least for zero mass). Earlier local and global well-posedness results can be found in \cite{Selberg2010b}. However the question of asymptotic behaviour is largely unknown. Some recent work in this direction has shown \emph{orbital stability} of the solitons \cite{Contreras2016, Pelinovsky2014}, but this leaves open the question of pointwise behaviour. In higher dimensions, $n>1$, the Thirring model is globally well-posed for small data and scatters to a linear solution in the scale invariant Sobolev space \cite{Bejenaru2014a, Bejenaru2016, Bournaveas2015}. Thus in the small data regime, the asymptotic behaviour is understood provided $n \not =1$.

 In the current article, our goal is present a first step towards understanding the pointwise asymptotic stability of the Dirac equation (\ref{eqn - thirring model}). More precisely, we adapt the arguments of Lindblad-Soffer \cite{Lindblad2005, Lindblad2005a, Lindblad2015} (see also \cite{Sunagawa2005, Sunagawa2006}), and show via energy estimates, together with an ODE argument, that the cubic nonlinearity causes an additional phase correction in the scattering behaviour. Our main result is as follows.

\begin{theorem}\label{thm - main}
Let $N \g 1$. There exists $\epsilon>0$ such that if the data satisfies
        $$ \big\| \lr{x}^{3+\frac{N}{2}}  f \big\|_{H^{N+4}} +\big\| \lr{x}^{3+\frac{N}{2}}   g \big\|_{H^{N+4}} \les \epsilon,$$
then in the exterior region $1\les t \les \lr{x}$ we have
    $$  |u(t, x)| + |v(t, x)| \lesa \lr{x}^{-\frac{N}{2}} \Big( \big\| \lr{x}^{3+\frac{N}{2}}  f \big\|_{H^{N+4}} +\big\| \lr{x}^{3+\frac{N}{2}}   g \big\|_{H^{N+4}} \Big).$$
On the other hand, when $t \g \lr{x}$, there exists bounded continuous functions $f_\pm$, such that
\begin{align*} u(t, x) &=  \frac{1}{\sqrt{t-x}} \Big( e^{  i \rho +  2i |f_+(\frac{x}{t})|^2 \ln(\rho)} f_+\big(\tfrac{x}{t}\big) + e^{  -i \rho +  2 i |f_-(\frac{x}{t})|^2 \ln(\rho)} f_-\big( \tfrac{x}{t} \big)\Big) + \mathcal{O}\Big( \frac{\rho^{-\frac{1}{2}}}{\sqrt{t-x}}\Big), \\
       v(t, x) &=    \frac{1}{\sqrt{t+x}} \Big( e^{  i \rho +  2i |f_+(\frac{x}{t})|^2 \ln(\rho)} f_+\big(\tfrac{x}{t}\big) - e^{  -i \rho +  2i  |f_-(\frac{x}{t})|^2 \ln(\rho)} f_-\big( \tfrac{x}{t} \big)\Big) + \mathcal{O}\Big( \frac{\rho^{-\frac{1}{2}}}{\sqrt{t+x}}\Big)\end{align*}
as $\rho = \sqrt{t^2 - x^2} \rightarrow \infty$.
\end{theorem}

We have made no attempt to optimise the decay or regularity assumptions on the data, and it is clear that the proof given below can be improved to somewhat sharpen the assumptions on the data. Alternatively, at a cost of complicating the proof, it should also be possible to obtain weaker decay conditions on the data by following the argument in the recent work of Stingo \cite{Stingo2015}. It is also worth noting that similar results should hold for (\ref{eqn - thirring model}) with more general cubic nonlinearities, however we do not consider this problem here. \\

The proof of Theorem \ref{thm - main} in the exterior region only exploits the additional decay of the Klein-Gordon equation when $t\les \lr{x}$ by using an argument of Klainerman \cite{Klainerman1993a}. In particular, the argument used here can also be used to remove the compact support assumptions from related works on the cubic Klein-Gordon equation \cite{Lindblad2005, Sunagawa2005, Delort2001}. However it is important to note that the work of Stingo \cite{Stingo2015} gives stronger results for the Klein-Gordon equation (in that it requires less decay on the data), than the weighted energy estimates approach used here.

In the interior region,  as in \cite{Lindblad2005, Lindblad2005a}, the proof of Theorem \ref{thm - main} proceeds by extracting the expected asymptotic behaviour, and use energy estimates on the hyperboloids $\{(t, x) | t^2 - x^2 = \rho^2\}$, together with an ODE formulation which reveals the precise asymptotic correction to the linear flow. This argument relies on the careful consideration of the linear Dirac flow, together with the precise structure of the nonlinearity in (\ref{eqn - thirring model}).

It is worth noting that, if we square the system (\ref{eqn - thirring model}), we obtain a nonlinear Klein-Gordon equation of the (schematic) form \begin{equation}\label{eqn - dirac as KG eqn}
  \Box \phi + \phi = \phi^3 + \phi^2 \p \phi
\end{equation}
thus it is tempting to try and deduce the asymptotic behaviour of $(u, v)$ from the corresponding result on the cubic Klein-Gordon equation given in \cite{Delort2001,  Lindblad2005, Stingo2015, Sunagawa2005}. However, the nonlinear terms in (\ref{eqn - dirac as KG eqn}) \emph{do not} satisfy the requirements needed to apply the previous results, and thus we have to work harder to obtain the asymptotic behaviour given in Theorem \ref{thm - main}.

\section{Exterior Region}

Given $T\g 1$, we consider the domain
    $$\mc{D}_T = \{ (t, x) \in \RR^{1+1} \,|\, \lr{x} \g t, \,\, 1\les t \les T\,\}$$
with boundary $S_T = \{  (T, x) \in \RR^{1+1} \,| \, \lr{x} \g T \}\cup \{ (\lr{x}, x) \,|\, \lr{x} \les T \,\} $. Define
    $$ E_{ext, T}(\phi, f)  =   \bigg( \int_{S_T} n^\alpha \, Q_{\alpha 0}[ \phi] \, f \, dx \bigg)^\frac{1}{2} $$
with $n^\alpha  \p_\alpha =  \p_t$ on $\{ \lr{x} \g T, t=T\}$,  $ n^\alpha \p_\alpha = \p_t + \frac{x}{\lr{x}} \p_x$ on $ \{ \lr{x} = t, 1\les t \les T\}$, and $Q_{\alpha \beta}$ is the Klein-Gordon energy momentum tensor
        $$ Q_{\alpha \beta} = \Re\Big( \p_\alpha \phi^\dagger \p_\alpha \phi  - \frac{1}{2} m_{\alpha \beta} \big( \p^\mu \phi^\dagger \p_\mu \phi - |\phi|^2 \big) \Big)$$
with the metric $m = \text{diag}(1, -1)$ and $\p_0 = \p_t$, $\p_1 = \p_x$. Note that
    $$ n^\alpha \, Q_{\alpha 0} = \begin{cases}
       \frac{1}{2} \big( |\p_t \phi|^2 + |\p_x \phi|^2 + |\phi|^2\big) \qquad &\lr{x} > T \text{ and } t=T,  \\
       \frac{1}{2} \big( |\p_t \phi|^2 + |\p_x \phi|^2 + |\phi|^2\big) + \frac{x}{t} \Re\big( (\p_t \phi)^\dagger \p_x \phi\big) &\lr{x} = t \text{ and } t<T
    \end{cases}
    $$
and hence $n^\alpha \,Q_{\alpha 0}  \g 0$. In particular $E_{ext, T}$ is well defined for any positive weight $f \g0$. Our goal is to prove the following weighted energy estimate (cf. \cite[Theorem 3]{Klainerman1993a}).

\begin{lemma}[Exterior Energy Estimates]\label{lem - ext energy estimate}
Let $1\les T < \infty$, $N \in \NN$ and for $0\les j \les N$ define the weights
    $$ w_j = ( t+ |x|)^{N-j} ( |x|-t+1)^j.$$
Then we have
    $$ \sum_{|I| \les N} E_{ext, T}(\p^I \phi, w_{|I|}) \lesa \sum_{|I| \les N} E_{ext, 1}(\p^I \phi, w_{|I|}) +  \sum_{|I| \les N} \int_1^T \bigg( \int_{ \lr{x} \g t} \big| \big( \Box +1\big) \p^I \phi \big|^2 w_{|I|} dx \bigg)^\frac{1}{2} dt.$$
\end{lemma}
\begin{proof}
We follow the argument of Klainerman \cite{Klainerman1993a}. An application of the divergence theorem gives for every $T\g 1$
    $$ \big( E_{ext, T}(\phi, f) \big)^2 = \big( E_{ext, 1}(\phi, f) \big)^2 + \int_{\mc{D}_T} Q_{\alpha 0}[\phi] \p^\alpha f dx dt + \int_{\mc{D}_T} \p^\alpha Q_{\alpha 0}[\phi] f dx dt. $$
Since $\p^\alpha Q_{\alpha 0} = \Re[ \p_t\phi^\dagger ( \Box \phi  + \phi)]$, the last integral can be estimated by
         $$ \int_{\mc{D}_T} \p^\alpha Q_{\alpha 0}[\phi] f dx dt = \int_1^T \int_{\lr{x} \g t}\Re[ \p_t\phi^\dagger ( \Box \phi  + \phi)] f dx dt \lesa \int_1^T E_{ext, t}(\phi, f) \Big( \int_{\lr{x} \g t} |\Box \phi + \phi|^2  f dx \Big)^\frac{1}{2} dt. $$
Consequently, the lemma will follow provided we can show that there exists constants $ c_j >0$ (depending only on $N$) such that
    $$ \sum_{0\les j \les N} c_{j} \sum_{|I| \les j}  Q_{\alpha 0}[\p^I \phi] \p^\alpha w_{j} \les 0.$$
To this end, we define the vector fields $e_\pm = \p_t \pm  \frac{x}{|x|} \p_x$ and observe that a computation gives the identity
    $$ Q_{\alpha 0}[\phi] \p^\alpha f = \frac{1}{4}\Big( e_+(f) \big( |e_-(\phi)|^2 + |\phi|^2 \big) + e_-(f) \big( |e_-(\phi)|^2 + |\phi|^2 \big) \Big). $$
Moreover, we can check that the weights $w_j$ satisfy $ e_+(w_N) = e_-(w_0) = 0$ and
    $$e_+( w_{j-1}) = 2 (N-j+1) \big( t+|x|\big)^{N-j} \big( |x|-t+1\big)^j =  - \frac{N-j+1}{j} e_-(w_{j})  \qquad 1\les j \les N.  $$
In particular we have $e_+(w_j) \g 0$ and $e_-(w_j) \les 0$ for every $0\les j \les N$. Therefore, since $ \sum_{|I| \les j} \big( |e_-(\p^I \phi)|^2 + |\p^I\phi|^2 \big) \les 2 \sum_{|I| \les j +1} |\p^I \phi|^2$, we deduce that
    \begin{align*}
      4\sum_{0\les j \les N} \sum_{|I| \les j} c_{j}  Q_{\alpha 0}[\p^I \phi] \p^\alpha w_{j} &=  \sum_{0\les j \les N} c_{j} \sum_{|I| \les j} e_+(f) \big( |e_-(\p^I \phi)|^2 + |\p^I \phi|^2 \big) + e_-(f) \big( |e_-(\p^I \phi)|^2 + |\p^I \phi|^2 \big) \\
      &\les \sum_{0\les j \les N} c_{j} \Big( 2 e_+(w_j) \sum_{|I| \les j+1} |\p^I \phi|^2 +  e_-(w_j) \sum_{ |I| \les j} |\p^I \phi|^2 \Big)\\
      &= \sum_{1\les j \les N} \big( 2c_{j-1} e_+(w_{j-1}) + c_j e_-(w_j) \big) \sum_{|I|\les j } |\p^I \phi|^2
    \end{align*}
which is less than zero provided we choose the constants $c_j$ such that
            $$ c_j = \frac{2(N-j +1)}{j} c_{j-1}, \qquad c_0 = 1.$$
\end{proof}

 If we have a function that satisfies $E_{ext, T}(\phi, w_j) < \infty$, then an application of Sobolev embedding\footnote{We use the observation that $\p_x w_j \lesa w_j$, together with the embedding $\|f\|_{L^\infty(x \g 0)} \lesa \| f \|_{H^1(x\g 0)}$ which follows from the standard inequality on $\RR$ by extending $f$ by reflection. } gives for any $\lr{x} \g T$ and $0 \les j \les N$
    $$  \big(| \phi|^2 w_{j}\big)(T, x) \lesa \big(E_{ext, T}(\phi, w_{j})\big)^2.$$
 We can use this estimate together with Lemma \ref{lem - ext energy estimate} and the formulation (\ref{eqn - dirac as KG eqn}) to deduce the exterior component of Theorem \ref{thm - main}. Let $Z = t \p_x + x \p_t$. A computation shows that
      \begin{align}
      \sum_{|I| \les N} \sum_{k=0}^3  \big|& \p^I Z^k \big( \phi^3 + \phi^2 \p \phi \big) \big|^2 w_{|I|}) \notag \\
      &\lesa \sum_{k_1, k_2, k_3=0}^3 \sum_{|I| \les N} w_{|I|} \sum_{J_1 + J_2 + J_3 = I}|\p^{J_1} Z^{k_1}\phi|^2 |\p^{J_2} Z^{k_2}\phi|^2 \big( |\p^{J_3} Z^{k_3}\phi|^2 + |\p^{J_3} \p Z^{k_3}\phi|^2\big) \notag\\
      &\lesa  \sum_{k_1, k_2, k_3=0}^3 \sum_{|I| \les N}  \sum_{J_1 + J_2 + J_3 = I}  \frac{w_{|I|}}{w_{|J_1|} w_{|J_2|} w_{|J_3|}} \notag \\
      &\qquad \qquad \times \big( E_{ext, T}( \p^{J_1}  Z^{k_1} \phi, w_{|J_1|}) E_{ext, T}(\p^{J_2} Z^{k_2} \phi, w_{J_2}) \big)^2 \big( |\p^{J_3} Z^{k_3} \phi|^2 + |\p^{J_3} \p Z^{k_3} \phi|^2\big) w_{J_3} \notag \\
      &\lesa t^{-2N} \sup_{\substack{|I|\les N \\
         0\les k \les 3}} \big(E_{ext, T}(\p^I Z^k\phi, w_{|I|}) \big)^4 \sum_{\substack{|I| \les N \\ 0 \les k \les 3}} \big(  |\p^I Z^k \phi|^2 +|\p \p^I Z^k \phi|^2\big) w_{|I|}. \label{eqn - cubic kg nonlinearity comp}
    \end{align}
 Let $(u, v)$ denote a solution to (\ref{eqn - thirring model}) and define
    $$ \mc{E}_{ext}(T) = \sum_{\substack{|I| \les N \\ 0 \les k \les 3}} \big( E_{ext, T}(\p^I Z^k u, w_{|I|}) + E_{ext, T}(\p^I Z^kv, w_{|I|}) \big).$$
 Then an application of Lemma \ref{lem - ext energy estimate}, together with the formulation (\ref{eqn - dirac as KG eqn}), the previous computation (\ref{eqn - cubic kg nonlinearity comp}), and the fact that $Z$ commutes with $\Box$,  gives
    \begin{align*}
      \mc{E}_{ext}(T) \lesa \mc{E}_{ext}(1) + \int_1^T t^{-2N} \mc{E}_{ext}(t) dt.
    \end{align*}
 Thus we obtain the following.

\begin{theorem}\label{thm - ext decay prop} Let $N \g 1$ and $(u, v)$ be a solution to (\ref{eqn - thirring model}). There exists a constant $\epsilon>0$ such that if the data satisfies
    $$ \mc{E}_{ext}(1) \les \epsilon,$$
then for every $T \g 1$ we have
    $$ \mc{E}_{ext}(T) \lesa \mc{E}_{ext}(1).$$
\end{theorem}
It is easy to check that this theorem gives the claimed decay rate in Theorem \ref{thm - main} in the exterior region $\lr{x} \g t$.

\section{Hyperbolic Coordinates}

We now turn to the more difficult inner region $t \g \lr{x}$. As in the previous works \cite{Delort2001, Lindblad2005, Sunagawa2005}, hyperbolic coordinates play a key role. Define the coordinates
    $$ t = \rho \cosh(y), \qquad x = \rho \sinh(y)$$
and let
    $$u(t,x) =  (\rho e^{-y})^{-\frac{1}{2}} U(\rho, y), \qquad v(t, x) = ( \rho e^{y})^{-\frac{1}{2}} V(\rho, y).$$
To control the solution in the interior region, we define the energy
    $$ \mc{E}_{int}(\rho) = \sum_{0\les k \les 3} \| \p_y^k U(\rho) \|_{L^2_y} + \| \p_y^k V(\rho) \|_{L^2_y}.$$
Arguing as in \cite[Section 7.6]{Hoermander1997}, \cite{Lindblad2005a}, the point wise identity
    $$ \sum_{0\les k \les 3} |\p^k_y U|^2 + |\p^k_y V|^2 \lesa \rho \cosh(y) \Big( \sum_{0\les k \les 3} |Z^k u|^2 + |Z^k v|^2 \Big)$$
implies that
        $$ \mc{E}_{int}(1) \lesa \lim_{T \rightarrow \infty} \mc{E}_{ext}(T) \lesa \mc{E}_{ext}(1).$$
Consequently, in view of the results in the previous section, we may assume that $\mc{E}_{int}(1)$ is small. The next step is to derive the equations satisfied by $(U, V)$. To this end, we note that since
     \begin{align*} (\p_t + \p_x) u  &= e^{ - \frac{y}{2}} \rho^{-\frac{1}{2}} \Big( \p_\rho U + \frac{1}{\rho} \p_y U \Big) \end{align*}
and
     \begin{align*} (\p_t - \p_x) v  &= e^{  \frac{y}{2}} \rho^{-\frac{1}{2}} \Big( \p_\rho V - \frac{1}{\rho} \p_y V \Big)  \end{align*}
the system (\ref{eqn - thirring model}) becomes
    \begin{align*}
      \p_\rho U + \frac{1}{\rho} \p_y U &= i V + i \frac{1}{\rho} |V|^2 U \\
      \p_\rho V - \frac{1}{\rho} \p_y V &= i U + i \frac{1}{\rho} |U|^2 V.
    \end{align*}
We require another version of the equation (\ref{eqn - thirring model}) to exploit the oscillatory behaviour of the solution. Define
    $$ \phi_\pm = e^{ \mp i \rho} ( U \pm V)$$
and $\phi = (\phi_+, \phi_-)$. Observe that
    $$e^{ \pm i \rho} \p_\rho \phi_\pm =   \big(\p_\rho U - i V \big)  \pm \big( \p_\rho V - i U\big)  =  \frac{1}{\rho} \Big( i |V|^2 U  \pm i |U|^2 V - \p_y \big( U \mp V\big) \Big).$$
Consequently we see that $\phi_\pm$ satisfies
    $$ \p_\rho \phi_\pm  + e^{ \mp 2 i \rho} \frac{1}{\rho} \p_y \phi_\mp = \frac{1}{\rho} i F_\pm $$
with
    $$ F_\pm = e^{ \mp i \rho} \big( |V|^2 U \pm |U|^2 V \big).$$
To compute $F_\pm$ in terms of $\phi_\pm$, we start by observing that
$$ |V|^2 U \pm |U|^2 V = \pm \big( U \pm V\big)^\dagger UV =  \pm \frac{1}{2} \big( U \pm V\big)^\dagger \Big( (U+V)^2 - (U-V)^2 \Big) $$
which implies that
    $$ \pm 2 F_\pm =  e^{ \mp i \rho} \big( e^{\pm i \rho} \phi_\pm \big)^\dagger \Big(  e^{ 2 i \rho} \phi_+^2 - e^{ - 2 i \rho} \phi_-^2 \Big)  .$$
Rearranging this then gives
    $$  2 F_\pm = |\phi_\pm|^2 \phi_\pm - e^{ \mp 4 i \rho} \big( \phi_\pm^\dagger \phi_\mp\big) \phi_\mp. $$
This has the important implication that we may write our equation as
     \begin{align} \p_\rho \phi_\pm &= \frac{i}{2 \rho} |\phi_\pm|^2 \phi_\pm  - \frac{1}{\rho}\Big( e^{\mp 2i \rho} \p_y \phi_\mp + \frac{i}{2} e^{ \mp 4 i \rho} (\phi^\dagger_\pm \phi_\mp) \phi_\mp \Big)  \notag \\
            &= \frac{i}{2 \rho} |\phi_\pm|^2 \phi_\pm + \p_\rho S_\pm + R_\pm \label{eqn - phi pm eqn}
     \end{align}
where
    $$S_\pm = \frac{-1}{\mp i \rho} e^{ \mp i \rho } \p_y \phi_\mp + \frac{-1}{\mp 8 \rho} e^{\mp 4 i \rho} (\phi_\pm^\dagger \phi_\mp) \phi_\mp$$
and
    $$R_\pm = \frac{1}{\mp i} e^{ \mp i \rho} \p_\rho\Big( \frac{1}{\rho} \p_y \phi_\mp \Big) + \frac{1}{\mp 8} e^{ \mp i \rho} \p_\rho \Big( \frac{1}{\rho} (\phi_\pm^\dagger \phi_\mp) \phi_\mp \Big). $$
The idea being that $R_\pm$ should be integrable, and thus can be considered a remainder term. On the other hand, the $\p_\rho S_\pm$ is not (absolutely) integrable, but can be absorbed into the left hand side. The remaining non-resonant term $|\phi_\pm|^2 \phi_\pm$ cannot be handled in this manner, and thus leads to the phase correction in the asymptotic behaviour. \\

\section{Interior Region}

Define
        $$M(\rho) = \sup_y \big( |U|^2 + |V|^2 \big)^\frac{1}{2} = 2 \sup_y \big( |\phi_+|^2 + |\phi_-|^2 \big)^\frac{1}{2}.$$
Our goal is to prove the following.
\begin{lemma}
There exists $\epsilon>0$ such that, if $\mc{E}_{int}(1) \les \epsilon$, then we have the global bound
        $$ \sup_{\rho \g 1} M(\rho) \lesa \mc{E}_{int}(1).$$
\end{lemma}
\begin{proof}
Fix $T>0$. A continuity argument shows that it is enough to prove that, provided we take $\epsilon>0$ sufficiently small, there exists a constant $C^*>0$ such that
    $$ \sup_{1\les \rho \les T} M \les 2 C^* \mc{E}_{int}(1) \qquad \Longrightarrow \qquad \sup_{1 \les \rho \les T} M \les C^* \mc{E}_{int}(1).$$
If we take the derivative of the energy $\mc{E}_{int}$, we obtain
    \begin{align*}
      \frac{1}{2} \p_\rho \mc{E}_{int}^2 &=  \sum_{0 \les k \les 3} \int_{\RR} \Re\Big[  (\p_y^kU)^\dagger\Big(  - \frac{1}{\rho} \p_y^{k+1} U + i \p^k_yV\Big) + (\p_y^kV)^\dagger\Big( \frac{1}{\rho} \p_y^{k+1} V + i \p_y^k U\Big) \Big] dy  \\
        & \qquad \qquad \qquad + \frac{1}{\rho} \int_{\RR} \Re\Big[ i (\p_y^k U)^\dagger \p_y^k \big( |U|^2 V\big) + i (\p_y^k V)^\dagger \p_y^k \big( |V|^2 U)\Big] dy \\
        &= L + \frac{1}{\rho} N.
    \end{align*}
To control the linear component $L$, we simply observe that
    $$ \int_\RR \Re\big[ (\p_y^k U)^\dagger \p_y^{k+1} U + (\p_y^k V) \p_y^{k+1} V \big] dy = \int_{\RR} \p_y \big( |\p_y^k U|^2 + |\p_y^k V|^2 \big) dy = 0$$
and
    $$   (\p_y^k U)^\dagger \p^k_y V + (\p_y^k V)^\dagger \p^k_y U = 2 \Re\big[ (\p_y^k U)^\dagger \p_y^k V \big]$$
which implies that $L=0$. On the other hand, an application of H\"older together with the product inequality for Sobolev spaces  gives $ N (\rho)\lesa M^2(\rho) \mc{E}_{int}(\rho)^2$. The  assumed bound on $M(\rho)$ and $\mc{E}_{int}(1)$ then implies that
    $$ \p_\rho \mc{E}_{int}^2(\rho)  \les  \frac{C}{\rho} M^2(\rho) \mc{E}_{int}(\rho)^2 \les \frac{C (2C^* \epsilon)^2}{\rho}  \mc{E}_{int}(\rho)^2$$
for some constant $C>0$. Therefore, letting $\delta = \frac{1}{2} C (2 C^* \epsilon)^2$ denote half the constant in the above inequality, we deduce that
        \begin{equation}\label{eqn - energy est rough}  \mc{E}_{int}(\rho) \les \mc{E}_{int}(1) e^{ \frac{1}{2}  C(2 C^* \epsilon)^2 \ln \rho} = \mc{E}_{int}(1) \rho^\delta. \end{equation}
Thus the energy $\mc{E}_{int}(\rho)$ is slowly growing. This bound is not enough on its own to control the solution, and we need to use the precise structure of the nonlinear terms to deduce the bound on $M(\rho)$. More precisely, a computation using (\ref{eqn - phi pm eqn}) shows that
     \begin{equation}\label{eqn - eqn for phi^2} \p_\rho \big( |\phi_\pm|^2 - 2 \Re( \phi_\pm^\dagger S_\pm) - |S_\pm|^2 \big) = 2 \Re\Big( - \frac{i}{2 \rho} |\phi_\pm|^2 \phi^\dagger_\pm S_\pm + R^\dagger_\pm S_\pm + \phi^\dagger_\pm R_\pm \Big). \end{equation}
The definitions of $S_\pm$ and $R_\pm$ implies that
    $$ |S_\pm| \lesa \frac{1}{\rho} ( |\p_y \phi| + |\phi|^3) \lesa \frac{1}{\rho} \mc{E}_{int} ( 1 + \mc{E}_{int}^2)$$
and
      \begin{align*} |R_\pm| &\lesa \frac{1}{\rho^2} \big( |\p_y \phi| + |\phi|^2 \big) + \frac{1}{\rho} \big( |\p_y \p_\rho \phi| + |\phi|^2 |\p_\rho \phi|\big) \\
      &\lesa  \frac{1}{\rho^2} \big( |\p_y \phi| + |\phi|^2 \big) + \frac{1}{\rho^2} \big( |\p_y^2 \phi| + |\phi|^2 |\p_y \phi| + |\phi|^2 |\p_y \phi| + |\phi|^5\big) \\
      &\lesa \frac{1}{\rho^2} \mc{E}_{int} ( 1 + \mc{E}_{int})^5.
      \end{align*}
Thus an application of the bound (\ref{eqn - energy est rough}) gives
    $$ |S_\pm| \lesa \mc{E}_{int}(1)  \rho^{ -1 + 3 \delta}, \qquad |R_\pm| \lesa \mc{E}_{int}(1) \rho^{ -2 + 5 \delta}$$
(here we assumed that $\mc{E}_{int}(1) \lesa 1$, and $\rho\g 1$). Therefore, provided we assume that $ 0< \delta < \frac{1}{10}$, we may integrate the equation (\ref{eqn - eqn for phi^2}) to deduce that there exists a constant $C$ (independent of $C^*$, $\epsilon$, and $\phi$) such that
    $$ M(\rho) \les M(1) + C \mc{E}_{int}(1).$$
Consequently, assuming that $C^*> C + 1$ and choosing $\epsilon  \ll \frac{1}{C^*}$, we obtain
        $$ M(\rho) \les C^* \mc{E}_{int}(1)$$
as required.
\end{proof}
\begin{remark}\label{rem - limit of |phi|^2}
  The proof of the above lemma shows something more. Namely, that there exists functions $ a_\pm(y) \g 0$ such that
        $$ \lim_{\rho \rightarrow \infty} |\phi_\pm|^2 = a_\pm$$
  and moreover,
        $$ \big| |\phi_\pm|^2 - a_\pm \big| \lesa  \rho^{ -1 + 5 \delta} \les \rho^{ - \frac{1}{2}} $$
  (by perhaps choosing $\epsilon$ slightly smaller).
\end{remark}

\subsection{Asymptotic Behaviour}
Our goal is to determine what happens to $(U, V)$ for large $\rho$. Recall that we have the equation
    $$ \p_\rho \phi_\pm = \frac{i}{2\rho} |\phi_\pm|^2 \phi_\pm + \p_\rho S_\pm + R_\pm $$
as well as the bounds
    $$ |S_\pm| \lesa \rho^{-\frac{1}{2}}, \qquad |R_\pm| \lesa \rho^{ - \frac{3}{2}}, \qquad \big| |\phi_\pm|^2 - a_\pm \big| \lesa \rho^{-\frac{1}{2}}.$$
If we multiply the equation for $\phi_\pm$ with the integrating factor $e^{  \frac{i}{2} a_\pm \ln(\rho)}$, we deduce that
    $$ \p_\rho \big( e^{ -\frac{i}{2} a_\pm \ln(\rho)} \phi_\pm - e^{ -\frac{i}{2} a_\pm \ln(\rho)} S_\pm \big) = \frac{i}{2\rho}\big( |\phi_\pm|^2 - a_\pm\big) \phi_\pm  + \frac{i}{2 \rho} a_\pm S_\pm + e^{ - \frac{i}{2} a_\pm \ln(\rho)} R_\pm. $$
The previous bounds imply that the right hand side is integrable, and hence
        $$ \lim_{\rho \rightarrow \infty} e^{ - \frac{i}{2} a_\pm \ln(\rho)} \phi_\pm = \sigma_\pm(y) $$
exists, uniformly in $y \in \RR$. Clearly we must have $|\sigma_\pm|^2 = a_\pm$, and consequently we can write
    $$ \phi_\pm(\rho, y) = e^{ \frac{i}{2}  |\sigma_\pm(y)|^2 \ln(\rho)} \sigma_\pm(y) + \mathcal{O}(\rho^{-\frac{1}{2}}).$$
In terms of $(U, V)$, this becomes
     \begin{align*} U(\rho, y) &= \frac{1}{2} \Big( e^{  i \rho +  \frac{i}{2} |\sigma_+(y)|^2 \ln(\rho)} \sigma_+(y) + e^{  -i \rho +  \frac{i}{2} |\sigma_-(y)|^2 \ln(\rho)} \sigma_-(y)\Big) + \mathcal{O}(\rho^{-\frac{1}{2}}),\\
       V(\rho, y) &= \frac{1}{2} \Big( e^{  i \rho +  \frac{i}{2} |\sigma_+(y)|^2 \ln(\rho)} \sigma_+(y) - e^{  -i \rho +  \frac{i}{2} |\sigma_-(y)|^2 \ln(\rho)} \sigma_-(y)\Big) + \mathcal{O}(\rho^{-\frac{1}{2}}).\end{align*}
If we return back to our original functions $(u, v)$, this is
     \begin{align*} u(t, x) &=  \frac{1}{2\sqrt{t-x}} \Big( e^{  i \rho +  \frac{i}{2} |\sigma_+(y)|^2 \ln(\rho)} \sigma_+(y) + e^{  -i \rho +  \frac{i}{2} |\sigma_-(y)|^2 \ln(\rho)} \sigma_-(y)\Big) + \mathcal{O}\Big( \frac{\rho^{-\frac{1}{2}}}{\sqrt{(t-x)}}\Big), \\
       v(t, x) &=    \frac{1}{2\sqrt{t+x}} \Big( e^{  i \rho +  \frac{i}{2} |\sigma_+(y)|^2 \ln(\rho)} \sigma_+(y) - e^{  -i \rho +  \frac{i}{2} |\sigma_-(y)|^2 \ln(\rho)} \sigma_-(y)\Big) + \mathcal{O}\Big( \frac{\rho^{-\frac{1}{2}}}{\sqrt{(t+x)}}\Big).\end{align*}
Defining the functions $f_{\pm}(s) = \frac{1}{2} \sigma_{\pm} \big( \frac{1}{2} \ln( 1 + s) - \frac{1}{2} \ln(1-s) \big)$ (which implies that $ 2f_\pm(\frac{x}{t}) = \sigma_\pm(y)$) we then obtain Theorem \ref{thm - main}.\\

\noindent \textit{Acknowledgements.} The authors would like to thank A. Stingo for helpful discussions regarding the work \cite{Stingo2015}.

% ----------------------------------------------------------------
\bibliographystyle{amsplain}
%\bibliography{references}

\providecommand{\bysame}{\leavevmode\hbox to3em{\hrulefill}\thinspace}
\providecommand{\MR}{\relax\ifhmode\unskip\space\fi MR }
\providecommand{\MRhref}[2]{%
  \href{http://www.ams.org/mathscinet-getitem?mr=#1}{#2}
}
\providecommand{\href}[2]{#2}

\end{document}